\documentclass[12pt]{amsart}
\usepackage{amssymb,amsmath,amsfonts,epsfig,latexsym,amsthm,comment,color}

\newtheorem{thm}{Theorem}[section]

\newtheorem{prop}[thm]{Proposition}
\newtheorem{lem}[thm]{Lemma}
\newtheorem{cor}[thm]{Corollary}
\newtheorem*{KCC}{Klyachko's Compatibility Condition}
\theoremstyle{definition}

\newtheorem{remark}[thm]{Remark}

\newtheorem{example}[thm]{Example}

\def\A{\ensuremath{\mathbf{A}}}

\def\P{\ensuremath{\mathbf{P}}}
\def\Q{\ensuremath{\mathbf{Q}}} 
\def\R{\ensuremath{\mathbf{R}}}
\def\Z{\ensuremath{\mathbf{Z}}}

\def\D{\mathcal{D}}

\def\F{\mathcal{F}}
\def\G{\mathcal{G}}
\def\L{\mathcal{L}}

\def\cR{\mathcal{R}}

\def\O{\ensuremath{\mathcal{O}}}

\def\k{\mathbf{k}}

\DeclareMathOperator{\Bl}{Bl}

\DeclareMathOperator{\Cl}{Cl}

\DeclareMathOperator{\Hom}{Hom}

\DeclareMathOperator{\Proj}{Proj}

\DeclareMathOperator{\Sym}{Sym}

\DeclareMathOperator{\Grass}{Grass}

\def\<{\ensuremath{\langle}}
\def\>{\ensuremath{\rangle}}

\begin{document}

\title[Cox rings of toric bundles]{Cox rings and pseudoeffective cones of projectivized toric vector bundles}

\author{Jos\'e Gonz\'alez}
\author{Milena Hering}    
\author{Sam Payne}
\author{Hendrik S\"u{\ss}}

\begin{abstract}
We study projectivizations of a special class of toric vector bundles that includes cotangent bundles, whose associated Klyachko filtrations are particularly simple.  For these projectivized bundles, we give generators for the cone of effective divisors and a presentation of the Cox ring as a polynomial algebra over the Cox ring of a blowup of a projective space along a sequence of linear subspaces.  As applications, we show that the projectivized cotangent bundles of some toric varieties are not Mori dream spaces and give examples of projectivized toric vector bundles whose Cox rings are isomorphic to that of $\overline{M}_{0,n}$.
\end{abstract}

\maketitle

\section{Introduction}

Projectivizations of toric vector bundles over complete toric varieties are a large class of rational varieties that have interesting moduli and share some of the pleasant properties of toric varieties and other Mori dream spaces.  Hering, Musta\c{t}\u{a}, and Payne showed that their cones of effective curves are polyhedral and asked whether their Cox rings are indeed finitely generated \cite{tvbpositivity}.  For rank two bundles an affirmative answer is given in \cite{HausenSuess10, Gonzalez10} or can be derived from the results of  \cite{Knop93}.


Here we apply general results of Hausen and S\"u{\ss} on Cox rings for varieties with torus actions to give a presentation of the Cox ring for certain projectivized toric vector bundles as a polynomial algebra over the Cox ring of the blowup of projective space along a collection of linear subspaces.  The question of finite generation for the Cox rings of these blowups is completely understood when the collection of linear subspaces consists of finitely many points in very general position, through work of Mukai \cite{Mukai04}, and Castravet and Tevelev \cite{CastravetTevelev06} in connection with Hilbert's fourteenth problem.

Let $\k$ be an algebraically closed field, and let $X$ be a smooth projective toric variety of dimension $d$ over $\k$, corresponding to a fan $\Sigma$ with $n$ rays.  Throughout, we use $r$ to denote the rank of a vector bundle on $X(\Sigma)$. By a \emph{toric vector bundle} on $X$ we mean a vector bundle admitting an action of the dense torus $T$ in $X$ that is linear on fibers and compatible with the action on the base. By the \emph{projectivization} of a toric vector bundle we mean the bundle of rank one quotients.

\begin{thm} \label{thm:CoxRing}
Suppose $\k$ is uncountable, $n > r \geq d$, and $\frac{1}{r} + \frac{1}{n-r} \leq \frac{1}{2}$.  Then there is a non-split toric vector bundle $\F$ of rank $r$ on $X(\Sigma)$ such that the Cox ring of the projectivization $\P(\F)$ is not finitely generated.
\end{thm}

\noindent In particular, on any smooth projective toric surface corresponding to a fan with at least nine rays, there is a rank three toric vector bundle whose projectivization is not a Mori dream space.  The bundles that we construct in the proof of Theorem~\ref{thm:CoxRing} are of a special form; in Klyachko's classification, they correspond to collections of filtrations in which each filtration contains at most one nontrivial subspace, which is of codimension one, and this arrangement of hyperplanes is in very general position.  The inequality in the theorem is sharp; if $\frac{1}{r} + \frac{1}{n-r} > \frac{1}{2}$ and the hyperplanes are in general position, then the projectivization of any such bundle is a Mori dream space. See Corollary~\ref{cor:MDS}.

\begin{remark}
The techniques used to prove Theorem~\ref{thm:CoxRing} give more information than just whether or not a Cox ring is finitely generated.  In Section~\ref{sec:CoxRings} we give presentations for the Cox rings of certain projectivized toric vector bundles as algebras over Cox rings of blowups of projective spaces along linear subspace arrangements.  As one special case, we produce an example of a vector bundle on a toric surface whose projectivization has the same Cox ring as $\overline{M}_{0,n}$.   See Example~\ref{ex:Kapranov}.
\end{remark}

\begin{remark}
If $\P(\F)$ is a projectivized bundle whose Cox ring is not finitely generated, it may still happen that the section ring of the tautological quotient line bundle $\O(1)$ on $\P(\F)$ is finitely generated.  However, Theorem~\ref{thm:CoxRing} implies that there also exist toric vector bundles $\F'$ such that the section ring of $\O(1)$ on $\P(\F')$ is not finitely generated.

Suppose $\P(\F)$ is a projectivized toric vector bundle on $X(\Sigma)$ whose Cox ring is not finitely generated, and let $\L_1, \ldots, \L_k$ be line bundles that positively generate the Picard group of $X(\Sigma)$.   Then the section ring of $\O(1)$ on the projectivization of $\F' = \F \oplus \L_1 \oplus \cdots \oplus \L_k$ is not finitely generated.  So Theorem~\ref{thm:CoxRing} gives negative answers to Questions 7.1 and 7.2 of \cite{tvbpositivity}.  
\end{remark}

A necessary, but not sufficient, condition for a projective variety to be a Mori dream space is that its pseudoeffective cone be polyhedral.  In many of the examples covered by Theorem~\ref{thm:CoxRing}, it is unclear whether this condition holds.  However, by choosing the toric variety carefully, with an even larger number of rays, we produce examples of projectivized toric vector bundles whose pseudoeffective cones are not polyhedral.

\begin{thm} \label{thm:EffCones}
Suppose $\k$ is uncountable, $n-d > r \geq d$, and $\frac{1}{r} + \frac{1}{n- d- r} \leq \frac{1}{2}$, and assume there is some cone $\sigma \in \Sigma$ such that every ray of $\Sigma$ is contained in either $\sigma$ or $-\sigma$.  Then there is a non-split toric vector bundle $\F$ of rank $r$ on $X(\Sigma)$ such that the pseudoeffective cone of $\P(\F)$ is not polyhedral.
\end{thm}

\noindent One can construct examples of toric varieties satisfying the hypotheses of Theorem~\ref{thm:EffCones} through sequences of iterated blowups of $(\P^1)^d$, as in Example~\ref{ex:dim2}, below.

The constructions used to prove Theorems~\ref{thm:CoxRing} and \ref{thm:EffCones} involve choosing bundles that are very general in their moduli spaces.  However, by choosing the fan sufficiently carefully, one gets examples of smooth projective toric varieties in characteristic zero whose projectivized cotangent bundles are not Mori dream spaces.  For these examples, the bundle is determined by the combinatorial data in the fan.

\begin{thm} \label{thm:CotangentBundle}
Suppose $d \geq 3$ and the characteristic of $\k$ is not two or three.  Then there exists a smooth projective toric variety $X(\Sigma')$ of dimension $d$ over $\k$ such that the Cox ring of the projectivized cotangent bundle on $X(\Sigma')$ is not finitely generated.
\end{thm}

\noindent In this respect, cotangent bundles behave quite differently from tangent bundles, since the Cox ring of the projectivization of the tangent bundle on any smooth toric variety is finitely generated \cite[Theorem~5.8]{HausenSuess10}.  So, Theorem~\ref{thm:CotangentBundle} shows that there are toric vector bundles $\F$ such that $\P(\F)$ is a Mori dream space, but the projectivized dual bundle $\P(\F^\vee)$ is not.

\begin{remark}
In Theorems~\ref{thm:CoxRing} and \ref{thm:EffCones}, we assume the field is uncountable in order to choose a configuration of points in very general position in the projective space $\P^{r-1}$.  Examples constructed by Totaro in his work on Hilbert's 14th Problem over finite fields \cite{Totaro08} show that this restriction on the cardinality of the field is not necessary in some cases.  For instance, to prove these theorems in the special case where $r$ is three, it is enough to find a configuration $S$ of nine points in $\P^2(\k)$ such that $\Bl_S \P^2$ contains infinitely many $-1$-curves, and Totaro constructed such configurations over $\Q$ and over $\mathbf F_p$, for $p > 3$.
\end{remark}

We conclude the introduction with an example of a projectivized rank three bundle on an iterated blowup of $\P^1 \times \P^1$ at seven points whose effective cone agrees with the effective cone of $\P^2$ blown up at nine very general points, and hence is not polyhedral.

\begin{example}\label{ex:dim2}
Let $X(\Sigma)$ be the toric variety obtained by first blowing up one of the toric fixed points on $\P^1 \times \P^1$, then blowing up both of the toric fixed points in the exceptional divisor, and then blowing up all four of the torus fixed points in the new exceptional divisors.  The corresponding fan is as shown.

\bigskip

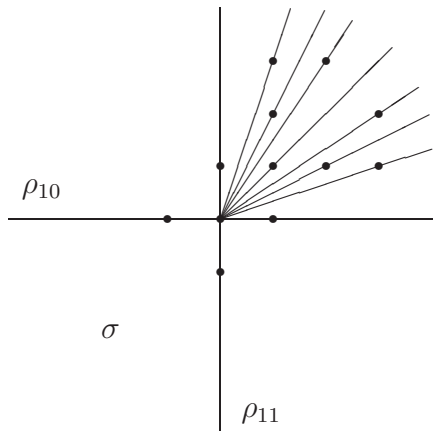
\begin{center}
\begin{picture}(160,160)(-80,-80)
\put(0,0){\circle*{2.5}}
\put(-80,0){\line(1,0){160}}
\put(0,-80){\line(0,1){160}}

\put(-20,0){\circle*{2.5}}
\put(0,-20){\circle*{2.5}}
\put(20,0){\circle*{2.5}}
\put(0,20){\circle*{2.5}}
\put(20,20){\circle*{2.5}}
\put(20,40){\circle*{2.5}}
\put(40,20){\circle*{2.5}}
\put(60,20){\circle*{2.5}}
\put(20,60){\circle*{2.5}}
\put(40,60){\circle*{2.5}}
\put(60,40){\circle*{2.5}}

\put(0,0){\line(1,1){65}}
\put(0,0){\line(1,2){39.5}}
\put(0,0){\line(1,3){26.6}}
\put(0,0){\line(2,3){50}}
\put(0,0){\line(2,1){79}}
\put(0,0){\line(3,1){80}}
\put(0,0){\line(3,2){75}}

\put(-75,10){$\rho_{10}$}
\put(8,-75){$\rho_{11}$}

\put(-45,-45){$\sigma$}

\end{picture}
\end{center}

\bigskip

\noindent Note that every ray of the fan is contained in either the cone $\sigma$ spanned by $\rho_{10}$ and $\rho_{11}$, or in $-\sigma$, and $\frac{1}{3} + \frac{1}{11-2-3} = \frac{1}{2}$.   So $X(\Sigma)$ satisfies the hypotheses of Theorems~\ref{thm:CoxRing} and \ref{thm:EffCones}.

Let $F$ be a three-dimensional vector space, and define filtrations
\begin{equation*}
F^{\rho_i}(j) = \left \{ \begin{array}{ll} F & \mbox{ for } j \leq 0, \\
								  F_i  & \mbox{ for } j = 1, \\
								  0 & \mbox{ for } j > 1,
								  \end{array} \right.
\end{equation*}
where $F_1, \ldots, F_9$ are two-dimensional subspaces in very general position, and $F_{10}$ and $F_{11}$ are zero.  By \cite{Klyachko90} these filtrations give rise to a toric vector bundle $\F$ on $X(\Sigma)$, see also Section~\ref{sec:prelim}. The subspaces $F_1, \ldots, F_9$ correspond to a set $S = \{p_1, \ldots, p_9\}$ of nine points in very general position in the projective plane $\P_F$ of one-dimensional quotients of $F$.  Our first main construction, in Section~\ref{sec:CoxRings}, shows that the Cox ring of $\P(\F)$ is canonically isomorphic to a polynomial ring in two variables over the Cox ring of the blowup $\Bl_S \P_F$ of the plane at this set of points.  Furthermore, in Section~\ref{sec:EffCones} we give an isomorphism of class groups $\Cl(\P(\F)) \xrightarrow{\sim} \Cl(\Bl_S \P_F)$ that takes $\O(1)$ to the pullback of the hyperplane class of $\P_F$, and the class of $\P(\F|_{D_{\rho_i}})$ to the class of the exceptional divisor $E_i$, for $i = 1, \ldots, 9,$ and show that this isomorphism induces an identification of the effective cones of the two spaces.  Therefore, the pseudoeffective cone of $\P(\F)$, like the pseudoeffective cone of $\Bl_S \P_F$, is not polyhedral, and $\P(\F)$ is not a Mori dream space.
\end{example}

\noindent {\textbf{Acknowledgments.}}  We are grateful to the mathematics department of Free University, Berlin, where much of this collaboration took place, and our hosts K. Altmann and C. Haase, for their warm and generous hospitality.  The second author was partially supported by NSF grant DMS 1001859.  The third author was partially supported by NSF grant DMS 1068689.  We thank A.~Bayer,  J. Hausen, A. Laface, and M. Musta\c{t}\u{a} for useful discussions, as well as J. Tevelev, B. Totaro, and the referee for helpful comments on an earlier draft of this paper.

\section{Preliminaries} \label{sec:prelim}

We work over an uncountable field $\k$ of arbitrary characteristic with the exception of the proof of Theorem~\ref{thm:CotangentBundle}, where we restrict to characteristic not two or three.

Let $T$ be a torus of dimension $d$, with character lattice $M$.  Let $X(\Sigma)$ be a toric variety with dense torus $T$, and let $\rho_1, \ldots, \rho_n$ be the rays of $\Sigma$.  We write $v_j$ for the primitive generator in $N = \Hom(M,\Z)$ of the ray $\rho_j$, and $D_{\rho_j}$ for the corresponding prime $T$-invariant divisor in $X(\Sigma)$.

Suppose $\F$ is a toric vector bundle of rank $r$ on $X(\Sigma)$. 
The Klyachko filtrations associated to $\F$ are decreasing filtrations of the fiber $F$ over the identity $1_T$, indexed by the rays of $\Sigma$,
\[
\cdots \supset F^{\rho_j}(k-1) \supset F^{\rho_j}(k) \supset F^{\rho_j}(k+1) \supset \cdots,
\]
and characterized by the following property.  If $U_\sigma$ is the torus-invariant affine open subvariety of $X(\Sigma)$ corresponding to a cone $\sigma$ in $\Sigma$, then the
torus $T$ acts on $H^0(U_\sigma,\F)$ by $(ts)(x) = t(s(t^{-1}x))$. 
If $u$ is a character of the torus, then the space of isotypical sections
\[
H^0(U_\sigma, \F)_u = \{ s \in H^0(U_\sigma, \F) \ | \ ts  
=\chi^{u}(t) s \text{ for all $t\in T$}\}
\]
injects into $F$, by evaluation at $1_T$, and the image is
\[
F^\sigma_u = \bigcap_{\rho_j \preceq \sigma} F^{\rho_j}(\<u, v_j\>).
\]
In particular, if $F^{\rho_j}(0) = F$ for all $j$ then the space of $T$-invariant sections of $\F$ is canonically isomorphic to $F$.

The Klyachko filtrations satisfy the following compatibility condition.

\begin{KCC} 
For each maximal cone $\sigma \in \Sigma$, there are lattice points $u_1, \ldots, u_r \in M$ and a decomposition into one-dimensional subspaces $F = L_1 \oplus \cdots \oplus L_r$ such that 
\[
F^{\rho_j}(k) = \bigoplus_{\< u_i, v_j \> \, \geq \, k} L_i,
\]
 for each $\rho_j \preceq \sigma$ and all $k \in \Z$.
\end{KCC}

\noindent The bundle $\F$ can be recovered from the family of filtrations $\{ F^{\rho_j}(k)\}$, and the induced correspondence between toric vector bundles and finite dimensional vector spaces with compatible families of filtrations gives an equivalence of categories.  See the original paper \cite{Klyachko90} or the summary in \cite[Section~2]{moduli} for details.

We write $\P(\F)$ for the projective bundle $\Proj(\Sym (\F))$ of rank one quotients of $\F$, and
\[
\pi: \P(\F) \rightarrow X(\Sigma)
\]
for its structure map.  The fiber of $\P(\F)$ over $1_T$ is the projective space $\P_F$ of one-dimensional quotients of $F$.  If $F'$ is a linear subspace of $F$ then
\[
\P_{F/F'} \subset \P_F
\]
is a projective linear subspace of codimension equal to the dimension of $F'$.

Following the usual convention, we write $\O(1)$ for the tautological quotient bundle on $\P(\F)$, which is relatively ample with respect to $\pi$, and $\O(m)$ for its $m$th tensor power.

For our primary examples in this paper, we will focus on bundles whose filtrations are especially simple, and in particular those satisfying
\begin{equation}
\label{eq:standardform}
F^{\rho_j}(k) = \left \{ \begin{array}{ll} F & \mbox{ for } k \leq 0, \\
								  F_j  & \mbox{ for } k = 1, \\
								  0 & \mbox{ for } k > 1,
								  \end{array} \right.  \tag{*}
\end{equation}
where $F_j$ is either $0$ or a subspace of $F$ of dimension at least two, and all of the nonzero $F_j$ are distinct.

One reason for working with a bundle given by filtrations satisfying (\ref{eq:standardform}) is that the  $T$-invariant global sections of $\O(m)$ on $\P(\F)$, and their orders of vanishing along the divisors $\pi^{-1}(D_{\rho_j})$, are particularly easy to understand.  See Lemmas~\ref{lem:sections} and \ref{lem:hypersurfaces}.

\begin{remark}
  \label{rem:Compatibility}
Suppose $\{F^{\rho_j}(j)\}$ is a collection of filtrations satisfying (\ref{eq:standardform}) in which all of the $F_j$ are hyperplanes.  Using the fact that $X(\Sigma)$ is smooth, one checks that Klyachko's compatibility condition for a cone $\sigma$ is satisfied for some $u_1 \ldots, u_r$ if and only if the hyperplanes $F_j$ for $\rho_j \preceq \sigma$ intersect transversely.  Since at most $r$ hyperplanes can meet transversely in a vector space of rank $r$, the condition $r \geq d$ appearing in Theorems~\ref{thm:CoxRing} and \ref{thm:EffCones} is necessary for such a collection of filtrations to define a toric vector bundle.  If the $F_j$ are chosen in general position, then the condition $r \geq d$ is also sufficient.
\end{remark}

\section{Torus quotients and Cox rings} \label{sec:CoxRings}

Let $X$ be a smooth variety whose divisor class group is finitely generated and torsion free.  Choose divisors $D_1, \ldots, D_k$ whose classes form a basis for the class group $\Cl(X)$.  Then the Cox total coordinate ring of $X$ is 
\[
\cR(X) = \bigoplus_{(m_1, \ldots, m_k) \in \Z^k} H^0(X, \O(m_1D_1 + \cdots + m_k D_k)),
\]
with the natural multiplication map of global sections.  See \cite{HuKeel00} for further details and a discussion of the special properties of Mori dream spaces, those varieties whose Cox rings are finitely generated.  If $X_0 \subset X$ is an open subvariety whose complement has codimension at least two, then $\Cl(X_0)$ and $\cR(X_0)$ are naturally identified with $\Cl(X)$ and $\cR(X)$, respectively.

\begin{remark}
Cox rings can be defined in greater generality, for possibly singular and nonseparated prevarieties whose class groups are finitely generated, but may contain torsion \cite{Hausen08}.  Cox rings of smooth and separated varieties with torsion free class groups suffice for all of the purposes of this paper, although we do consider nonseparated quotients in some generalizations of Theorem~\ref{thm:polynomial} presented in Section~\ref{sec:generalizations}.
\end{remark}

Our main technical result is a description of the Cox ring of certain projectivized toric vector bundles as a polynomial ring over the Cox ring of a blowup of projective space.  Let $S$ be a finite set of projective linear subspaces of $\P_F$ and let $S'$ be the set of intersections of subspaces in $S$.  Say $L_1, \ldots, L_s$ are the elements of $S'$.  We write $\Bl_{S'} \P_F$ for the space obtained by blowing up first the points in $S'$, then the strict transforms of the lines in $S'$, then the strict transforms of the two-dimensional subspaces in $S'$, and so on.  We write $E_i$ for the exceptional divisor in $\Bl_{S'} \P_F$ dominating $L_i$, and define
\[
\Bl_S \P_F = \Bl_{S'} \P_F \smallsetminus \Big( \bigcup_{L_i \not \in S} E_i \Big)
\]

\begin{example}
  Let $x_1$, $x_2$, and $x_3$ be non-collinear points in $\P^3$, and let $L_{ij}$ be the line through $x_i$ and $x_j$, and set
  \[
  S = \{ x_1, L_{12}, L_{13}, L_{23} \}.
  \]
  Then $S' = S \cup \{x_2, x_3\}$ and $\Bl_S \P^3$ is the space obtained by blowing up first the points $x_1$, $x_2$, and $x_3$, and then the strict transforms of the lines $L_{12}$, $L_{13}$, and $L_{23}$, and then removing the exceptional divisors over $x_2$ and $x_3$. 
  \end{example}
  
 Our main technical result can now be stated as follows.  Let $\F$ be a toric vector bundle on a complete toric variety $X$ given by filtrations satisfying the condition (\ref{eq:standardform}) discussed in Section~\ref{sec:prelim}.  After renumbering, say the $F_i$ are distinct linear subspaces for $i \leq s$, and $F_j$ is zero for $s < j \leq n$.  Let
\[
S = \{ \P_{F/F_1}, \ldots, \P_{F/F_s} \}
\]
be the set of projective linear subspaces in $\P_F$ corresponding to $F_1, \ldots, F_s$.

\begin{thm} \label{thm:polynomial}
The Cox ring $\cR(\P(\F))$ is isomorphic to a polynomial ring in $n-s$ variables over $\cR(\Bl_S \P_F)$.
\end{thm}

Our proof of Theorem~\ref{thm:polynomial} will be an application of the following presentation of Cox rings for certain varieties with torus actions.

\begin{prop} \label{prop:HS}
Let $X$ be a smooth variety such that $H^0(X,\mathcal{O}_X^{*}) = \k^*$ and $\Cl(X)$ is free, and let $T$ be a torus acting on $X$.
Suppose $D_1, \ldots, D_h$ are irreducible divisors in $X$ with positive dimensional generic stabilizers, $T$ acts freely on $X \smallsetminus (D_1 \cup \cdots \cup D_h)$, and the geometric quotient is a smooth variety $Y$ with free class group.  Then $\cR(X)$ is isomorphic to a polynomial ring in $h$ variables over $\cR(Y)$.
\end{prop}

\begin{proof}
This is the special case of \cite[Theorem~1.1]{HausenSuess10}, where $X$ is smooth, the $T$-action on the complement of $D_1 \cup \cdots \cup D_h$ is free, and the geometric quotient $Y$ is separated, with torsion free class group.  Although stated in the case where X is complete, the proof of that theorem is also given under the assumption that $H^0(X, \O_X^*)= \k^*$ and $\Cl(X)$ is free, which is what we use here. 
\end{proof}

We prove Theorem~\ref{thm:polynomial} by constructing a dominant rational map
\[
\varphi: \P(\F) \dashrightarrow \Bl_S \P_F,
\]
and producing open sets $U \subset U'$ in $\P(\F)$ with the following properties:
\begin{enumerate}
\item The complement of $U'$ has codimension 2 in $\P(\F)$.
\item There are $n-s$ irreducible divisors in $U'$ with positive dimensional generic stabilizers, the complement of these divisors is $U$, and $T$ acts freely on $U$.
\item The restriction $\varphi|_U$ is regular and a geometric quotient.
\item The complement of $\varphi(U)$ has codimension 2 in $\Bl_S \P_F$.
\end{enumerate}
\noindent To see that Theorem~\ref{thm:polynomial} follows from the existence of such a map, first note that class groups, global invertible functions and Cox rings are all invariant under the removal of sets of codimension 2.  Therefore, (1) implies that $H^0(U', \mathcal{O}_{U'}^{*}) = \k^*$, the class group $\Cl(U')$ is free, and $\cR(U') \cong \cR(\P(\F))$.  Then, by Proposition~\ref{prop:HS}, (2) and (3) imply that $\cR(U')$ is isomorphic to a polynomial ring in $n-s$ variables over $\cR(\varphi(U))$.  Finally, (4) gives $\cR(\varphi(U)) \cong \cR(\Bl_S \P_F)$.  Therefore, (1)-(4) together imply that $\cR(\P(\F))$ is isomorphic to a polynomial ring in $n-s$ variables over $\cR(\Bl_S \P_F)$.

We construct the birational map $\varphi$ and open sets $U$ and $U'$ as follows.  There is a unique dominant, $T$-invariant rational map
\[
\psi: \P(\F) \dashrightarrow \P_F
\]
that restricts to the identity on the fiber $\P_F$ over $1_T$.  Over the dense torus $T$, this map takes a point $x$ in the fiber over $t$ to $t^{-1} \cdot x$.  One can also describe $\psi$ as the rational map associated to the $T$-invariant linear series $H^0(\P(\F), \O(1))_0$; the sections of $\O(1)$ are canonically identified with sections of $\F$, and evaluation at $1_T$ maps $H^0(X, \F)_0$ isomorphically onto $F$.  Alternatively, $\psi$ can be constructed directly from the $T$-invariant sections of $\F$, which generate all fibers over $T$, following the general construction in \cite[Example~6.1.15]{PAG1}.  

Since $\psi$ is dominant, it induces a $T$-invariant rational map $\varphi$ to $\Bl_S \P_F$.  To prove Theorem~\ref{thm:polynomial}, we produce open sets $U \subset U'$ in $\P(\F)$ satisfying (1)-(4) with respect to $\varphi$.

We write $x_i$ for the distinguished point in the codimension one orbit corresponding to $\rho_i$; see \cite[Section~2.1]{Fulton93}.  The fiber of $\F$ over $x_i$ is canonically isomorphic to $F_i \oplus F/F_i$; this is the eigenspace decomposition for the action of the one parameter subgroup corresponding to the primitive generator of $\rho_i$, which is the stabilizer of $x_i$.  Let $Z_i$ be the projective linear subspace corresponding to $\P_{F_i}$ in the fiber of $\P(\F)$ over $x_i$.  Let $W_i$ be the projective linear subspace corresponding to $\P_{F/F_i}$ in the fiber over $1_T$.

We now define $U'$ to be the complement in $\P(\F)$ of the following closed subsets.
\begin{itemize}
\item The preimages of the $T$-invariant closed subsets of codimension 2 in $X$.
\item The torus orbit closures $\overline{T \cdot Z_i}$, for $F_i \neq 0$.
\item The torus orbit closures $\overline{T \cdot W_i}$, for $F_i \neq 0$.
\end{itemize}
Note that the condition (\ref{eq:standardform}) says that $F_i$ has dimension at least 2 and codimension at least 1 whenever it is nonzero.  Therefore, every component of the complement of $U'$ has codimension at least 2.

This choice of closed subsets is closely related to the intedeterminacy locus of $\varphi$.  On the fiber over $1_T$, this map is the birational inverse of the blowup morphism from $\Bl_S \P_F$ to $\P_F$, so its indeterminacy locus is the discriminant, which is the union of the $W_i$.  The closures $\overline{T \cdot W_j}$ may meet the fiber over $x_i$, and these intersections are also in the indeterminacy locus of $\varphi$ because the indeterminacy locus is closed.  In the special case where $i = j$, the intersection of $\overline{T \cdot W_i}$ with the fiber over $x_i$ is the linear subspace $\P_{F/F_i}$.  Now, $\varphi$ maps the fiber over $x_i$ into the exceptional divisor over $\P(F/F_i)$, via the canonical rational map
\[
\P(F_i \oplus F/ F_i) \dashrightarrow \P_{F/F_i} \times \P_{F_i},
\]
which is regular away from the linear subspaces $W_i$ and $Z_i$.  In particular, $Z_i$ is the only remaining indeterminacy locus of $\varphi$ in the fiber over $x_i$.  Therefore, after removing the preimage of the codimension 2 strata in $X$, the open set $U'$ is simply the locus where $\varphi$ is regular.

For $i = s+1, \ldots, n$, the subspace $F_i$ is zero.  Then the one parameter subgroup corresponding to the primitive generator of $\rho_i$ acts trivially on $\P(\F|_{O_{\rho_i}})$.  Let $U$ be the complement in $U'$ of these $n-s$ irreducible divisors with positive dimensional stabilizers.

We claim that $T$ acts freely on $U$.  Over the dense torus, $T$ acts freely on the base.  Over a codimension one orbit $O_{\rho_i}$, the stabilizer on the base is the one-parameter subgroup corresponding to the primitive generator of $\rho_i$.  Because the eigenspace decomposition of the fiber of $\F$ over $x_i$ is $F_i \oplus F/ F_i$, with the one parameter subgroup acting by scaling on $F_i$ and trivially on $F/F_i$, this subgroup acts freely away from the $T$-orbits of the linear subspaces $\P_{F_i}$ and $\P_{F/F_i}$, both of which are in the complement of $U'$ and hence of $U$.  Therefore, $T$ acts freely on $U$.  

To prove the theorem, it remains to show that $\varphi|_U$ is a geometric quotient and the image $\varphi(U)$ has codimension 2 in $\Bl_S \P_F$.  We first treat the special case where the fan $\Sigma$ has only a single ray $\rho$.  Let $U_\rho$ be the toric variety corresponding to a single ray $\rho$, and let $\F$ be the toric vector bundle on $U_\rho$ given by the filtration
\[
F^{\rho}(k) = \left \{ \begin{array}{ll} F & \mbox{ for } k \leq 0, \\
								  F_\rho  & \mbox{ for } k = 1, \\
								  0 & \mbox{ for } k > 1,
								  \end{array} \right.
\]
where $F_\rho$ is a proper subspace of dimension at least two.  Then $\F$ splits canonically as a sum $\F = \F_\rho \oplus \F / \F_\rho$, where $\F_\rho$ is the toric subbundle with fiber $F_\rho$ over $1_T$.  Let $Z$ be the projective linear subspace $\P_{F_\rho}$ in the fiber over $x_\rho$, and let $W$ be the projective linear subspace $\P_{F/F_\rho}$ in the fiber over $1_T$.

\begin{prop}  \label{prop:oneray}
The torus $T$ acts freely on the open set
\[
\P(\F) \smallsetminus (\overline{T \cdot Z} \cup \overline{T \cdot W})
\]
with geometric quotient $\Bl_{W} \P_F$, and the preimage of $O_\rho$ under the structure map surjects onto the exceptional divisor over $W$.
\end{prop}

\begin{proof}
The open set $\P(\F) \smallsetminus (\overline{T \cdot Z} \cup \overline {T \cdot W})$ is the set denoted $U$ in the discussion above, and hence $T$ acts freely.  We use a toric computation to compute the geometric quotient.

The projectivization of any toric vector bundle $\G$ of rank $r$ on $U_\rho$ is isomorphic to a toric variety.  The toric variety is canonical, but the isomorphism depends on the choice of a splitting of the fiber over $1_T$,
\[
G = L_1 \oplus \cdots \oplus L_r,
\]
 satisfying Klyachko's compatibility condition.  Fix such a splitting.  For $1 \leq j \leq r$, define the integer $n_j = \max \{ k \ | \ G^{\rho}(k) \mbox{ contains } L_j \}$.  Let $\sigma$ be the cone in $N_\R \times \R^r$ spanned by the standard vectors $(0,e_1), \ldots, (0,e_r)$ and
\[
\widetilde v_\rho = (v_\rho, n_1 e_1 + \cdots + n_r e_r),
\]
where $v_\rho$ is the primitive generator of $\rho$, and let $\Delta$ be the fan in $N_\R \times (\R^r / (1,\ldots,1))$ whose maximal cones are projections of the facets of $\sigma$ that contain $\widetilde v_\rho$.  Then there is a natural isomorphism from $\P(\G)$ to $X(\Delta)$ taking $\P(\G|_{O_\rho})$ to the torus invariant divisor corresponding to the image of $\widetilde v_\rho$ and the codimension one projectivized subbundle corresponding to $L_j$ to the torus invariant divisor corresponding to the image of $(0, e_j)$.  
For further details on this construction, see \cite[pp. 58--59]{Oda88}.

We now apply the preceding general construction to our particular bundle $\F$.  The decomposition into line bundles induces a decomposition of  $F=L_1 \oplus \cdots \oplus L_r$ into  one-dimensional coordinate subspaces.  After relabeling we may assume that 
$F_\rho = L_1 \oplus \cdots \oplus L_{k}$ and so $\widetilde{v}_\rho
= (v_\rho,e_1 + \cdots + e_{k})$. Then  the complement $\P(\F) \smallsetminus (\overline{T \cdot Z} \cup \overline{T \cdot W})$, with its induced toric structure, corresponds to the fan $\Delta'$ in $N_\R \times (\R^r/ (1, \ldots, 1))$ obtained by removing from the fan for $\P(\F)$ the cones containing either $\widetilde{v}_\rho$ and  $(0,e_{k+1}), \ldots, (0,e_r)$ (corresponding to $\overline{T \cdot Z}$), or all of $(0,e_1), \ldots, (0,e_{k})$ (corresponding to $\overline{T \cdot W}$).

The projection of $N_{\R} \times \R^r/(1, \ldots, 1)$ onto $\R^r/(1, \ldots, 1)$ induces a map of fans from $\Delta'$ 
to the fan of the blow up of $\P_F$ along $\P_{F/F_\rho}$.  This map of fans satisfies the conditions of 
\cite[Proposition~3.2]{ACampoNeuenHausen99}, and hence the corresponding morphism of toric varieties is a geometric quotient. 
\end{proof}

We now apply the special case treated above, where the fan consists of a single ray $\rho$, to prove the general case.

\begin{proof}[Proof of Theorem~\ref{thm:polynomial}]
By the discussion following Proposition~\ref{prop:HS}, it remains to show that $\varphi|_U$ is a geometric quotient and the complement of $\varphi(U)$ has codimension 2 in $\Bl_S \P_F$.  The property of being a geometric quotient is local on the base.  For each $i$, let $U_i$ be the complement in $\Bl_S \P_F$ of the exceptional divisors over $W_j$ for $j \neq i$.  

We claim that the preimage of $U_i$ under $\varphi$ is the preimage in $U$ of the $T$-invariant affine open set $U_{\rho_i}$, under the structure map $\pi$.  Indeed, by Proposition~\ref{prop:oneray}, the rational map $\varphi$ takes the generic point of $\P(\F)|_{O_{\rho_j}}$ to the generic point of the exceptional divisor over $W_j$.  The part of $U$ that lives over $T$ maps into every $U_i$, but for the parts of $U$ over codimension one orbits of $X$, only the part over $O_{\rho_i}$ maps into $U_i$.  This proves the claim.

The union of the sets $U_i$ cover all but a codimension 2 locus in $\Bl_S \P_\F$, so it only remains to show that the restriction of $\varphi$ to the preimage of $U_i$ is a geometric quotient.  Again, this follows from the local computation in Proposition~\ref{prop:oneray}, because $U_i$ is just the complement of the codimension 2 loci given by the strict transforms of the $W_j$ for $j \neq i$ in $\Bl_{W_i} \P_F$, and the restriction of $\varphi$ to $\varphi^{-1}(U_i)$ is the restriction of the geometric quotient onto $\Bl_{W_i} \P_F$ described in Proposition~\ref{prop:oneray}.
\end{proof}

\begin{proof}[Proof of Theorem~\ref{thm:CoxRing}]
	Let  $S$ be a subset of $s$  very general points of $\P_F \cong \P^{r-1}$ such that $s\geq r+2 + \frac{4}{r-2}$. 
	Then 
$\cR(\Bl_S \P_F)$ is not finitely generated, see \cite{Mukai04}. 
Let $X$ be any smooth toric variety of dimension at most $\dim(\P_F)$ 
with at least $s$ rays. Then 
by Remark~\ref{rem:Compatibility} there exists a vector bundle 
$\F$ on $X$  satisfying (\ref{eq:standardform}) such that the nonzero 
$F_i$ correspond to the points $p_i$ in $S$. 
The conclusion then 
follows from Theorem~\ref{thm:polynomial}. 
\end{proof}

\begin{remark}
	Note that the isomorphism of Theorem~\ref{thm:polynomial}
	is not  an isomorphism of graded rings. However, the 
	pull back by the quotient map $\varphi$ constructed in the 
	proof of Theorem~\ref{thm:polynomial} induces a 
	group homomorphism $\varphi^* \colon \mathrm{Cl}(\Bl_S(\P_F) 
	\to \mathrm{Cl}(\P(\F))$. Letting $\deg(x_i) = [\pi^{-1}(D_{\rho_i})]$ 
	for $s+1\leq i\leq n$, we obtain a $\mathrm{Cl}(\P(\F))$-grading of the polynomial ring in $n-s$ variables 
	over the Cox ring of $\Bl_S(\P_F)$  such that the isomorphism of 
	Theorem~\ref{thm:polynomial} is graded. 
\end{remark}


\begin{cor}  \label{cor:MDS}
Suppose $\F$ is given by filtrations satisfying \emph{(\ref{eq:standardform})} with the $F_i$ being hyperplanes in general position.  If $\frac{1}{r} + \frac{1}{n-r} > \frac{1}{2}$ then $\P(\F)$ is a Mori dream space.
\end{cor}

\begin{proof}
Suppose $\frac{1}{r} + \frac{1}{n-r} > \frac{1}{2}$. Then the blow up of $\P^{r-1}$ at $n$ points in general position is a Mori dream space \cite[Theorem~1.3]{CastravetTevelev06}, and then so is the blow up $\Bl_s \P^{r-1}$ of $\P^{r-1}$ at $s$ points in general position, where $s$ is the number of rays $\rho_j$ such that $F_j$ is nonzero. The corollary then follows immediately from Theorem~\ref{thm:polynomial}, which says that $\cR(\P(\F))$ is finitely generated over $\cR(\Bl_s \P^{r-1})$.
\end{proof}

\noindent If the points $p_1, \ldots, p_s$ are not in general position then $\P(\F)$ can be a Mori dream space, even when $\frac{1}{r} + \frac{1}{n-r} \leq \frac{1}{2}$.  For instance, if $p_1, \ldots, p_s$ are collinear then $\Bl_S \P_F$ is a rational variety with a torus action with orbits of codimension one, and hence is a Mori dream space \cite{EKW04, HausenSuess10, Ottem10}.  Also, if $p_1, \ldots, p_s$ lie on a rational normal curve, then $\Bl_S \P_F$ is a Mori dream space \cite[Theorem~1.2]{CastravetTevelev06}.

We conclude this section with the observation that the Cox ring of the blowup of projective space along an arbitrary arrangement of linear subspaces can be realized as the Cox ring of a projectivized toric vector bundle.

\begin{cor} \label{cor:arbitrary}
Let $S$ be an arbitrary arrangement of $n$ linear subspaces of codimension at least 2 in $\P_F$ and let $\Sigma$ be a fan with $n$ rays that defines a smooth projective toric surface.  Then there is a toric vector bundle $\F$ on $X(\Sigma)$ such that $\cR(\P(\F)) \cong \cR(\Bl_S \P_F)$.
\end{cor}

\begin{proof}
An arbitrary collection of filtrations of $F$ indexed by the rays of $\Sigma$ satisfies Klyachko's compatibility condition, because $X(\Sigma)$ is a smooth surface \cite[Example~2.3.4]{Klyachko90}.  Therefore, if $S = \{ \P_{F/F_1}, \ldots, \P_{F/F_n}\}$ then the filtrations
\[
F^{\rho_j}(k) = \left \{ \begin{array}{ll} F & \mbox{ for } k \leq 0, \\
								  F_j  & \mbox{ for } k = 1, \\
								  0 & \mbox{ for } k > 1,
								  \end{array} \right.
\]
satisfy (\ref{eq:standardform}) and determine a toric vector bundle on $X(\Sigma)$.  By Theorem~\ref{thm:polynomial}, the Cox ring $\cR(\P(\F))$ is isomorphic to $\cR(\Bl_S \P_F)$.
\end{proof}

\begin{example} \label{ex:Kapranov}
Let $S$ be the arrangement of all linear subspaces of codimension at least 2 spanned by subsets of a set of $r+1$ points in general position in $\k^r$.  Then Kapranov's construction \cite{Kapranov93} shows that $\Bl_S \P^{r-1}$ is isomorphic to the Deligne-Mumford moduli space $\overline{M}_{0,r+2}$.  Therefore, there is a toric vector bundle $\F$ on a smooth projective toric surface such that $\cR(\P(\F))$ is isomorphic to the Cox ring of $\overline{M}_{0,n}$.  It is not known whether this ring is finitely generated.
\end{example}

\section{Cotangent bundles}

In the previous section, we gave a presentation of Cox rings of some projectivized toric vector bundles as polynomial rings over Cox rings of certain blowups of projective space, and used this to give examples where Cox rings of projectivized toric vector bundles are finitely generated, where they are not finitely generated, and where they are isomorphic to the Cox ring of $\overline{M}_{0,n}$.  We now apply the same methods and results to study Cox rings of projectivized cotangent bundles of smooth projective toric varieties.

  By \cite{Klyachko90} the filtrations of cotangent bundles have the following form.
\begin{equation*}
\Omega^{\rho_j}(k) = \left \{ \begin{array}{ll} M \otimes \k & \mbox{ for } k \leq -1 \\
								  v_j^\perp  & \mbox{ for } k = 0, \\
								  0 & \mbox{ for } k > 0,
				  \end{array} \right.
\end{equation*}
If the fan does not contain any pair of opposite rays, then the filtrations for the twist of the cotangent bundle by the anticanonical line bundle satisfy (\ref{eq:standardform}). Since twisting by a line bundle does not change the projectivization, Theorem~\ref{thm:polynomial} shows that the Cox ring of the projectivized cotangent bundle is isomorphic to the Cox ring of $\Bl_S(\P_{M\otimes {\k}})$, where $S$ is the set of points $p_j$ corresponding to $v_j^{\perp}$.  The case where the fan does contain opposite rays is treated in Section~\ref{sec:repetitions}.

\begin{example} \label{ex:projcotangent}
For the cotangent bundle on projective space $\P^r$, the corresponding set $S$ consists of $r+1$ points in linearly general position in $\P^{r-1}$.    Then $\cR(\P(\Omega^1_{\P^r}))$ is identified with $\cR(\Bl_S \P^{r-1})$, which is isomorphic to the coordinate ring of the Grassmannian $\Grass(2,r+2)$ in its Pl\"ucker embedding; see \cite[Remark~3.9]{CastravetTevelev06}.  
\end{example}

\noindent Example~\ref{ex:projcotangent} is a special case of the Cox rings of wonderful varieties studied by Brion \cite{Brion07}.

We now give an example of a smooth projective toric threefold whose projectivized cotangent bundle is not a Mori dream space.  The construction uses a particularly nice configuration of nine points in $\Z^3$, due to Totaro, such that, for any field $\k$ of characteristic not two or three, the blowup of $\P^2(\k)$ at the corresponding nine $\k$-points is not a Mori dream space.  The proof of Theorem~\ref{thm:CotangentBundle} will be by induction on dimension, starting from this example.

\begin{example} \label{ex:Dim3}
In this example, we work over a field $\k$ of characteristic not two or three.  

The vectors
\[
\begin{array}{llll}
v_1=(0,0,1), & v_2=(0,1,0), & v_3=(1,1,1), & v_4=(-1,-2,-2)
\end{array}
\]
span the four rays of a unique complete fan $\Sigma_4$ in $\mathbf{R}^3$.  The corresponding toric variety $X(\Sigma_4)$ is isomorphic to $\P^3$.  Consider the vectors
\[
\begin{array}{llll}
v_5 = (1,1,2), & v_6 = (0,-1,1), & v_7 = (1,0,1), & v_8 = (1,-1,1), \\
 v_9 = (-1,-2,-1), & v_{10} = (-1,-1,0), & v_{11} = (-1,-1,1), & v_{12} = (-1,0,1), \\
 v_{13} = (-1,1,1), & v_{14} = (0,1,1),&
\end{array}
\]
and let $\Sigma_i$ be the stellar subdivision of $\Sigma_{i-1}$ along the ray spanned by $v_i$, for $5 \leq i \leq 14$.  For each such $i$, the vector $v_i$ is the sum of two or three of the $v_j$ that span a cone in $\Sigma_{i-1}$.  Therefore, the toric variety $X(\Sigma_i)$ is the blowup of $X(\Sigma_{i-1})$ at either a point or a torus invariant smooth rational curve.  In particular, if we set $\Sigma = \Sigma_{14}$, then the corresponding toric variety $X(\Sigma)$ is smooth and projective.  The twist $\F$ of the cotangent bundle on $X(\Sigma)$ by the anticanonical bundle $\O(D_{\rho_1} + \cdots + D_{\rho_{14}})$ is given by the vector space $F = \k^3$ with filtrations
\[
F^{\rho_i}(j) = \left \{ \begin{array}{ll} \k^3 & \mbox{ for } j \leq 0, \\
								  v_i^\perp  & \mbox{ for } j = 1, \\
								  0 & \mbox{ for } j > 1,
								  \end{array} \right.
\]
Since the characteristic of $\k$ is not two or three, the points $v_i^\perp$ are all distinct in $\P^2_\k$, and hence the filtrations satisfy (\ref{eq:standardform}).  Twisting by a line bundle does not change the projectivization, so Theorem~\ref{thm:polynomial} says that the Cox ring of the projectivized cotangent bundle of $X(\Sigma)$ is isomorphic to the Cox ring of $\Bl_S \P^2_{\k}$, where $S = \{ v_1^\perp, \ldots, v_{14}^\perp \}$.  The subset
\[
S' = \{ v_1^\perp, v_3^\perp, v_6^\perp, v_7^\perp, v_8^\perp, v_{11}^\perp, v_{12}^\perp, v_{13}^\perp, v_{14}^\perp \}
\]
is the complete intersection of two smooth cubics, and the Cox ring of $\Bl_{S'} \P^2_{\k}$ is not finitely generated \cite[Theorem~2.1,  Corollary~5.1 and Theorem~5.2]{Totaro08}.  It follows that $\Bl_S \P^2_{\k}$ is not a Mori dream space, and neither is the projectivized cotangent bundle of $X(\Sigma)$.
\end{example}

We use the following lemma on Cox rings of blowups of projective space at finitely many points contained in a hyperplane in the proof of Theorem~\ref{thm:CotangentBundle}.  Instances of this basic fact have appeared, including in \cite[Example~1.8]{HassettTschinkel04}.  However, lacking a suitable reference, we give a proof.

\begin{lem} \label{lem:hyperplane}
Let $S$ be a finite set of points contained in a hyperplane $H$ in $\P^d$, and assume $d > 2$.  Then the Cox ring of $\Bl_S \P^d$ is isomorphic to a polynomial ring in one variable over the Cox ring of $\Bl_S H$.
\end{lem}

\begin{proof}
Choose coordinates on $\P^d$ so that $H$ is a coordinate hyperplane, and let $\mathbf G_m$ act by scaling on the coordinate that cuts out $H$.  The action of $\mathbf G_m$ lifts to an action on $\Bl_S \P^d$, and we let $Y$ be the locus of fixed points of this action.  Then $\mathbf G_m$ acts freely on $\Bl_S \P^d \smallsetminus Y$, with quotient $\Bl_S H$.  The strict transform of $H$ is the only divisor contained in $Y$, so the lemma follows by applying Proposition~\ref{prop:HS}.
\end{proof}

\begin{proof}[Proof of Theorem~\ref{thm:CotangentBundle}]
Let $\k$ be a field of characteristic not two or three.  We must show that, for each dimension $d \geq 3$, there is a fan $\Sigma$ in $\R^d$ such that
\begin{enumerate}
\item The toric variety $X(\Sigma)$ is smooth and projective.
\item The hyperplanes in ${\k}^d$ perpendicular to the primitive generators of the rays of $\Sigma$ are distinct.
\item The Cox ring of $\Bl_S \P^{d-1}_{\k}$ is not finitely generated, where $S$ is the set of points corresponding to these hyperplanes.
\end{enumerate}
For $d = 3$, we have Example~\ref{ex:Dim3}, and we proceed by induction.

Suppose $\Sigma$ is a fan in $\R^d$ satisfying (1), (2), and (3).  Embed $\R^d$ as the last coordinate hyperplane in $\R^{d+1}$, and let $\Sigma'$ be the fan in $\R^{d+1}$ whose maximal cones are spanned by a maximal cone of $\Sigma$ together with either $(1, \ldots, 1)$ or $(1, \ldots, 1, -1)$.  The corresponding toric variety $X(\Sigma')$ is smooth and projective and, since the characteristic of $\k$ is not two, the hyperplanes in $\k^{d+1}$ perpendicular to the rays of $\Sigma'$ are distinct.  It remains to show that $\Sigma'$ satisfies (3).  Let $S'$ be the corresponding set of points in $\P^{d}_{\k}$.  Now $S'$ contains the subset $S$ of points corresponding to rays of $\Sigma$, and $S$ is contained in a hyperplane $H$.  By hypothesis, the Cox ring of $\Bl_S H$ is not finitely generated.  By Lemma~
\ref{lem:hyperplane}, it follows that $\Bl_{S} \P^d_{\k}$ is not a Mori dream space, and neither is $\Bl_{S'} \P^d$.   The theorem follows, since the Cox ring of the projectivized cotangent bundle of $X(\Sigma)$ is isomorphic to the Cox ring of $\Bl_{S'} \P^d_{\k}$, by Theorem~\ref{thm:polynomial}.
\end{proof}

\section{Pseudoeffective cones} \label{sec:EffCones}

In this section we prove Theorem \ref{thm:EffCones}. The techniques 
of the proof are independent 
from those of Section \ref{sec:CoxRings}. 

The pseudoeffective cone of a projective variety $X$ is the closure of the cone spanned by the classes of all effective divisors in the space of numerical equivalence classes of divisors $N^{1}(X)_\R = N^{1}(X) \otimes_\Z \R$. For projectivized toric vector bundles and for blowups of projective spaces at finite sets of points, linear equivalence and numerical equivalence coincide and then we identify $N^{1}(X)_\R$ and $\Cl(X)_\R = \Cl(X) \otimes_\Z \R$. 



Now we consider again a toric vector bundle $\F$ on a complete toric variety $X(\Sigma)$. Any effective divisor $D$ on $\P(\F)$ is linearly equivalent to a torus
invariant effective divisor; this can be seen by applying the Borel fixed-point
theorem to the torus orbit closure of the point $[D]$ in the Chow
variety of effective codimension 1 cycles on $\P(\F)$. So the pseudoeffective cone of $\P(\F)$ is the closure of the cone generated by classes of torus invariant prime divisors.  Note that every torus invariant prime divisor in $\P(\F)$ is either the preimage of a torus invariant prime divisor in $X(\Sigma)$ or surjects onto $X(\Sigma)$.  If a torus invariant prime divisor surjects onto $X(\Sigma)$ then it must be the closure of the torus orbit of its intersection with the fiber over the identity.  We write $\D_H$ for the closure of the torus orbit of a hypersurface $H$ in $\P_F$.  One key step toward understanding the pseudoeffective cone of $\P(\F)$ is to express the class of each such $\D_H$ as a linear combination of $\O(1)$ and the $\pi^{-1} (D_{\rho_i})$.  Such expressions may be somewhat complicated in general, but are relatively simple for bundles given by filtrations of the special form discussed in Section~\ref{sec:prelim}.  

Suppose the filtrations $\{F^{\rho_i}(j)\}$  associated 
to the vector bundle $\F$ satisfy condition (\ref{eq:standardform}) of Section~\ref{sec:prelim} and all proper subspaces $F^{\rho_i}(j) \subset F$ are distinct hyperplanes.

\begin{lem} \label{lem:sections}
Restriction to the fiber $\P_F$ gives an isomorphism from the space of $T$-invariant global sections of $\O(m)$ on $\P(\F)$ to $\Sym^m(F)$.
\end{lem}

\begin{proof}
For any bundle $\F$, global sections of $\O(m)$ on $\P(\F)$ are naturally identified with global sections of $\Sym^m \F$.  Now, $\Sym^m \F$ is a toric vector bundle, with fiber $\Sym^m F$ over $1_T$, and since the filtrations defining $\F$ satisfy (\ref{eq:standardform}), the filtrations defining $\Sym^m \F$ are given by
\[
\Sym^m F^{\rho_i}(j) = \left \{ \begin{array}{ll} \Sym^m F & \mbox{ for } j \leq 0, \\
								    \operatorname{Image}\left( \Sym^j F_i \otimes \Sym^{m-j} F \rightarrow \Sym^m F \right) & \mbox{ for } 1 \leq j \leq m, \\
								  0 & \mbox{ for } j > m.
								  \end{array} \right.
\]
The space of $T$-invariant sections of $\Sym^m \F$ is the intersection of all of these filtrations evaluated at zero, and the lemma follows, because $\Sym^m F^{\rho_i}(0)$ is $\Sym^m F$, for every ray $\rho_i$.
\end{proof}

Let $p_i$ be the point in $\P_F$ corresponding to the one-dimensional quotient $F/F_i$, whenever $F_i$ is nonzero.  We write $D_j$ for the $T$-invariant prime divisor $\pi^{-1}(D_{\rho_j})$ in $\P(\F)$.  

\begin{lem} \label{lem:hypersurfaces}
Let $H$ be a hypersurface of degree $m$ in $\P_F$, and let $m_i$ be the multiplicity of $H$ at $p_i$.  Then there is a linear equivalence 
\[
\D_H \, \sim \ \O(m) - \sum_i m_i (\pi^{-1} (D_{\rho_i})),
\]
where the sum is over those $i$ such that $F_i$ is nonzero.
\end{lem}

\begin{proof}
Let $h \in \Sym^m F$ be a defining equation for $H$.  Then $h$ corresponds to a torus invariant section $s$ of $\O(m)$ on $\P(\F)$, by Lemma~\ref{lem:sections}.  If $F_i$ is zero then $s$ does not vanish along $D_i$ and if $F_i$ is nonzero then $m_i$ is the largest integer such that $h$ is contained in the image of $\Sym^{m_i} F_i \otimes \Sym^{m- m_i} F$ in $\Sym^m F$. The one parameter subgroup corresponding to $v_i$ extends to an embedding of the affine line $\A^1$ in $X(\Sigma)$ meeting $D_{\rho_i}$ transversely at the image of zero.  After restricting the section $s$ to the preimage of $\A^1$, we must show that its order of vanishing along the preimage of zero is $m_i$.  The isotypical decomposition of the module of global sections of $\O(1)$ on the preimage of $\A^1$, for the action of the one-parameter subgroup corresponding to $v_i$, is exactly $\bigoplus_j F^{\rho_i}(j)$, and multiplication by the coordinate $x$ on $\A^1$ decreases degree by one.  The sections of $\O(m)$ are given by the $m$th symmetric power of this module, in which the image of $\Sym^{k} F_i \otimes \Sym^{m- k} F$ in $\Sym^m F$ appears in degree $k$, for nonnegative integers $k$.  It follows that the $T$-invariant section $s$ is equal to $x^{m_i}$ times a section that is nonvanishing along the preimage of zero, and hence vanishes to order $m_i$, as required.
\end{proof}

Now, we fix a maximal cone $\sigma$ and, after renumbering, we may assume $\sigma$ is spanned by $\rho_1, \ldots, \rho_d$.  Furthermore, for the remainder of the section we assume that
\[
F_i = 0, \mbox{ for } 1 \leq i \leq d.
\]
The class of $\O(1)$ and the classes of $D_{d+1}, \ldots, D_n$ form a basis for  $\Cl(\P(\F))$.

Let $f: \Bl_S \P_F \rightarrow \P_F$ be the blowup of $\P_F$ at the finite set of distinct points $\{ p_i\}$, corresponding to the nonzero $F_i$, for $i > d$.  Let $L$ be a hyperplane in $\P_F$, and let $E_i$ be the exceptional divisor over $p_i$.  Then $f^*L$ and $\{E_i\}$ together form a basis for $\Cl(\Bl_S \P_F)$.

We consider the linear map $\varphi^*: \Cl(\Bl_S \P_F)_\R \rightarrow \Cl(\P(\F))_\R$, taking $f^*L$ to $\O(1)$ and the class of $E_i$ to the class of $D_i$, for $i > d$.  If $H$ is a hypersurface of degree $m$ in $\P_F$ passing through $p_i$ with multiplicity $m_i$, then the class of the strict transform of $H$ in $\Bl_S \P_F$ is $f^*mL - \sum_i m_i E_i$.  So Lemma~\ref{lem:hypersurfaces} says that $\varphi^*$ maps the class of the strict transform of $H$ to the class of $\D_H$.

\begin{remark}
	One can show that the map $\varphi^*$ is the map 
	on class groups induced by the map $\varphi$ of the proof 
	of Theorem \ref{thm:polynomial}, see 
	\cite[Section 5]{GonzalezThesis}. 
	However, note that 
	Lemmas \ref{lem:sections} and \ref{lem:hypersurfaces} 
	give an independent 
	proof of the fact that we get a morphism of class groups, 
	without having to construct 
	the morphism $\varphi$. 
\end{remark}

\begin{prop} \label{prop:EffCones}
The pseudoeffective cone of $\P(\F)$ is generated by the image under $\varphi^*$ of the pseudoeffective cone of $\Bl_S \P_F$ together with the classes of those $D_i$ such that $F_i$ is zero.
\end{prop}

\begin{proof}
Every effective divisor on $\P(\F)$ is in the cone generated by the classes $\D_H$, for hypersurfaces $H$ in $\P_F$, and the classes $D_i$. On $\Bl_S \P_F$, every effective divisor is in the cone generated by the classes of the strict transforms of the hypersurfaces $H$ in $\P_F$, and the classes $E_i$. Now, the classes $D_{i}$ such that $F_i$ is nonzero are the images under $\varphi^*$ of the classes $E_{i}$, and Lemma~\ref{lem:hypersurfaces} says that the class of $\D_H$ is the image under $\varphi^*$ of the strict transform of the hypersurface $H$ in $\P_F$. Therefore, the cone of effective classes on $\P(\F)$ is equal to the cone generated by the image under $\varphi^*$ of the cone of effective classes on $\Bl_S \P_F$ together with the classes of those $D_i$ such that $F_i$ is zero.  The proposition follows by taking closures.
\end{proof}

\begin{proof}[Proof of Theorem~\ref{thm:EffCones}]
Let $\sigma$ be the cone spanned by $\rho_1, \ldots, \rho_d$, and choose the toric variety $X(\Sigma)$ so that each of the other rays $\rho_i$ is contained in $-\sigma$.  This can be accomplished, as in Example~\ref{ex:dim2}, by taking a suitable sequence of blowups of $(\P^1)^d$.  Choose the filtrations defining $\F$ so that $F_{d+1}, \ldots, F_n$ are distinct hyperplanes, and $F_i = 0$ for $i \leq d$. 

The choice of the filtrations ensures that $\varphi^*$ is an isomorphism on class groups, since it maps the basis elements $f^*L, E_{d+1}, \ldots, E_n$ for $\Cl(\Bl_S \P_F)$ to the basis elements $\O(1), D_{d+1}, \ldots, D_n$ for $\Cl(\P(\F))$, respectively.  Furthermore, the choice of the fan $\Sigma$ ensures that, for $i \leq d$, the divisor $D_{\rho_i}$ is linearly equivalent to an effective combination of the $D_{\rho_j}$, for $j > d$.  So the classes of $D_1, \ldots, D_d$ are in the cone spanned by the classes of $D_i$ for $i > d$, and hence are in the image under $\varphi^*$ of the pseudoeffective cone of $\Bl_S \P_F$.  Therefore, by Proposition~\ref{prop:EffCones}, the linear isomorphism $\varphi^*$ identifies the pseudoeffective cone of $\Bl_S \P_F$ with the pseudoeffective cone of $\P(\F)$.  If $F_{d+1}, \ldots, F_n$ are in very general position, then the inequalities on $r$ and $n$ imply that the pseudoeffective cone of $\Bl_S \P_F$ is not polyhedral \cite{Mukai04}, and the theorem follows.
\end{proof}

\begin{remark}
As in Corollary~\ref{cor:arbitrary}, a similar construction produces toric vector bundles $\F$ such that the effective cone of $\P(\F))$ is canonically isomorphic to the effective cone of $\Bl_S \P_F$, for an arbitrary arrangement $S$ of linear subspaces in $\P_F$.
\end{remark}

\section{Some generalizations}  \label{sec:generalizations}

The techniques developed here can also be applied more generally to describe Cox rings of toric vector bundles where the condition (\ref{eq:standardform}) is weakened to allow $F_i$ to appear for multiple steps in the Klyachko filtrations, where some of the $F_i$ are alowed to be $1$-dimensional, and where the subspaces are not necessarily distinct.  The results are similar to those in Section~\ref{sec:CoxRings}, only the presentations of the Cox rings are slightly more complex.

\subsection{Longer steps in the filtrations}
Consider a toric vector bundle $\F$ given by Klyachko filtrations of the form
\begin{equation*}
F^{\rho_j}(k) = \left \{ \begin{array}{ll} F & \mbox{ for } k \leq 0, \\
								  F_j  & \mbox{ for } 1\leq k \leq a_j, \\
								  0 & \mbox{ for } k > a_j,
								  \end{array} \right.  
\end{equation*}
for some positive integers $a_j$, and distinct linear subspaces $F_j \subsetneq F$ of dimension at least 2, for $j =1, \ldots, s$.  The bundles that satisfy the condition (\ref{eq:standardform}) are exactly those where each $a_j$ is equal to 1.  The Cox ring of $\P(\F)$ can be analyzed just as in Section~\ref{sec:CoxRings}, except that $T$ does not act freely on $U$; if $D_j$ denotes the preimage of $O_{\rho_j}$ in $U$, then $D_j$ has a stabilizer of order $a_j$.  In this case, the Cox ring of $\P(\F)$ is a finite extension of a polynomial ring over $\cR(\Bl_S \P_F)$ with a presentation of the form
\[
\cR(\P(\F)) \cong \cR(\Bl_S \P_F)[x_1, \ldots, x_n] / \langle 1_{E_j} - x_j^{a_j} \ | \ 1 \leq j \leq s \rangle,
\]
by \cite[Theorem~1.1]{HausenSuess10}.  Here, $1_{E_i}$ denotes the canonical section of the bundle $\O(E_i)$ associated to the exceptional divisor $E_i$ over $\P_{F/F_i}$.  It follows that $\cR(\P(\F))$ is finitely generated if and only if $\cR(\Bl_S \P_F)$ is finitely generated.

\subsection{$1$-dimensional subspaces}  \label{sec:1-dim}

We now discuss toric vector bundles given by Klyachko filtrations of the form (\ref{eq:standardform}), but where some $F_j$ are allowed to be 1-dimensional.  Consider the special case where the fan $\Sigma$ consists of a single ray $\rho$, and $\F$ is given by the filtration
\[
F^{\rho}(k) = \left \{ \begin{array}{ll} F & \mbox{ for } k \leq 0, \\
								  L  & \mbox{ for } k = 1, \\
								  0 & \mbox{ for } k > 1,
								  \end{array} \right.
\]
where $L$ is 1-dimensional.  The analysis of such a bundle is similar to that in Proposition~\ref{prop:oneray}, except that $\overline{T \cdot W}$ is a divisor.  Still, the torus $T$ acts freely on the toric variety $\P(\F) \smallsetminus \overline{T \cdot Z}$, and a toric computation shows that the geometric quotient exists as a nonseparated toric prevariety; it is $\P_F$ with the hyperplane $\P_{F/L}$ doubled.

Now, consider the general case, and let $S$ be the set of linear subspaces of $\P_F$ corresponding to the $F_j$ that have dimension at least 2.  Suppose the rays are numbered so that $F_1, \ldots, F_\ell$ are $1$-dimensional and the rest are not.  Then the analysis in the proof of Theorem~\ref{thm:polynomial} produces open subsets $U$ and $U'$ satisfying (1)-(4), except that the target of $\varphi$ is $\Bl_S \P_F$ doubled along the strict transforms of the hyperplanes $H_i = \P_{F/F_i}$ for $ 1 \leq i \leq \ell$.  Then \cite{HausenSuess10} gives a presentation of the Cox ring $\cR(\P(\F))$ as a polynomial ring in $n-s$ variables over
\begin{equation}
  \label{eq:dim-1-filt}
  \cR(\Bl_S \P_F)[x_1, \ldots, x_\ell, y_1, \ldots, y_\ell]/ \langle 1_{H_i} - x_i y_i \ | \ 1 \leq i \leq \ell \rangle,
\end{equation}
where $1_{H_i}$ is the canonical section of $\O(H_i)$.  Setting the $n-s$ free variables equal to zero and $y_1, \ldots, y_\ell$ equal to 1, one can obtain $\cR(\Bl_S \P_F)$ as a quotient of $\cR(\P(\F))$, and hence $\cR(\P(\F))$ is finitely generated if and only if $\cR(\Bl_S \P_F)$ is so.

\begin{example} Using the above observations, in \cite{HausenSuess10} the Cox ring of the projectivized tangent bundle of a toric variety was calculated as follows. Let $X$ be a toric variety associated to fan $\Sigma$ with rays $\rho_1, \ldots, \rho_n$ having $v_1, \ldots,v_n \in N$ as their primitive generators. By \cite{Klyachko90} the tangent bundle $T_X$ corresponds to the filtrations of the form (\ref{eq:standardform}) with $F^{\rho_j}=\k\cdot v_j \subset N \otimes \k$.
In particular, all the subspace are one-dimensional. Hence, the set $S$ is empty and $\cR(\Bl_S \P_F)=\cR(\P_F)$ is simply the polynomial ring $\Sym(F)$. The element $1_{H_j}$ can be identified with $v_j \in \Sym(F)$. 
If there are no opposite rays  in $\Sigma$, by the formula (\ref{eq:dim-1-filt}) we obtain \[\k[x_1,\ldots, x_n,y_1,\ldots,y_n]/\big\langle \textstyle \sum_i \lambda_i \cdot x_iy_i
\mid  \underline{\lambda} \in \k^n \text{, s.t. } \sum_i \lambda_i v_i = 0\big\rangle\] as the Cox ring of $\P(T_X)$. 
\end{example}

\subsection{Repetitions of subspaces and combinations}  \label{sec:repetitions}
If some subspace is repeated, so $F_i = F_j$ for some $i \neq j$, then the arguments in Section~\ref{sec:CoxRings} again go through, but the geometric quotient is nonseparated, with one copy of the exceptional divisor over $\P_{F/F_i}$ for each time that $F_i$ appears.  Again, this construction leads to a presentation of $\cR(\P(\F))$ as a finitely generated algebra over $\cR(\Bl_S \P_F)$ that is finitely generated if and only if $\cR(\Bl_S \P_F)$ is so.

These generalizations can be combined to give a presentation of the Cox ring of an arbitrary toric vector bundle for which the Klyachko filtrations contain at most one nontrivial subspace for each ray.

\begin{prop}
Let $\F$ be a toric vector bundle corresponding to Klyachko filtrations $\{F^\rho(j) \}$ such that at most one proper subspace of $F$ appears in each filtration, and let $S$ be the collection of linear subspaces of $\P_F$ corresponding to these proper subspaces.  Then $\cR(\P(\F))$ is finitely generated if and only if $\cR(\Bl_S \P_F)$ is so.
\end{prop}

\begin{remark}
It may be possible to carry through a similar analysis for more general toric vector bundles.  However, even when the fan consists of a single ray, if multiple proper subspaces occur in a single filtration then the torus quotients that appear are weighted blowups of projective space instead of ordinary blowups.  Since very little is known about Cox rings of weighted blowups of projective space, we have not considered such bundles in this work.
\end{remark}

\subsection{A bundle on the Losev-Manin moduli space}
We conclude with an example of a bundle on the Losev-Manin moduli space of pointed stable curves.

Let $v_0, \ldots, v_d$ be vectors that generate the rank $d$ lattice $N$ and sum to zero.  Then the fan $\Sigma$ whose nonzero cones are spanned by proper subsets of $\{v_0, \ldots, v_d \}$ corresponds to projective space $\P^d$, and the barycentric subdivision $\Sigma'$ is the normal fan of a permutahedron.  The corresponding toric variety is the Losev-Manin moduli space $\overline{L}_d$ of pointed stable curves studied in \cite{LosevManin00}.

Let $\F$ be the pullback of the cotangent bundle $\Omega_{\P^d}$ to the Losev-Manin moduli space $\overline{L}_d \cong X(\Sigma')$.  The rays of $\Sigma'$ are naturally indexed by the proper subsets of $\{0, \ldots, d\}$, where the primitive generator of the ray $\rho_I$ is
\[
v_I = \sum_{i \in I} v_i.
\]
The fiber of $\F$ over $1_T$ is canonically identified with $M \otimes_\Z \k$, and we write $M_I$ for the linear subspace perpendicular to the linear span of the $v_i$ for $i \in I$.  The Klyachko filtrations corresponding to $\F$ are then
\[
F^{\rho_I}(k) = \left \{ \begin{array}{ll} M \otimes \k & \mbox{ for } k \leq -1 \\
								  M_I  & \mbox{ for } k = 0, \\
								  0 & \mbox{ for } k > 0,
				  \end{array} \right.
\]
These filtrations almost satisfy (\ref{eq:standardform}), except that the subspaces $M_I$ corresponding to sets $I$ of size $d-1$ are 1-dimensional, and the last nonzero subspace appears in the wrong place in the filtration.  Tensoring with an appropriate line bundle puts the last nonzero subspace in the correct place in the filtration and does not change the projectivization.  Then, applying the computation for filtrations with 1-dimensional subspaces in Section~\ref{sec:1-dim} above, we find that $\cR(\P(\F))$ is a polynomial ring in $d+1$ variables over 
\[
\textstyle  
\cR(\Bl_S \P_F)[x_1, \ldots, x_{d+1 \choose 2}, y_1, \ldots, y_{d+1 \choose 2}] / \langle 1_{H_i} - x_i y_i \ | \ 1 \leq i \leq {d+1 \choose 2} \rangle,
\]
where $H_i$ runs over runs over all the hyperplanes $\P_{F/F_I}$ with index sets $I$ of size $d-1$.

Now, $S$ consists of all linear subspaces spanned by $d + 1$ points in general position in $\P_F \cong \P^{d-1}$.  As in Example~\ref{ex:Kapranov}, the blowup $\Bl_S \P_F$ is isomorphic to the Deligne-Mumford moduli space $\overline{M}_{0,d+2}$.

\begin{cor}
The projectivization of the pullback of the cotangent bundle on $\P^d$ to the Losev-Manin moduli space $\overline {L}_d$ is a Mori dream space if and only if $\overline{M}_{0,d+2}$ is a Mori dream space.
\end{cor}

\bibliography{math}
\bibliographystyle{amsalpha}

\end{document}